\documentclass[preprint,12pt]{amsart}
\usepackage{amsmath}
\usepackage{amssymb}
\usepackage[mathscr]{eucal}
\usepackage{graphicx}
\usepackage{amsmath}
\usepackage{amsfonts}
\usepackage{fullpage}
\usepackage{times}
\usepackage{bbm}
\usepackage{epsfig}
\usepackage[usenames]{color}
\usepackage[pagebackref]{hyperref}    
\newtheorem{definition}{Definition}

\newcommand{\ch}{ \mathsf{ch}}

\newcommand{\CF}{ \mathsf{CF}}
\newcommand{\SF}{ \mathsf{SF}}
\newtheorem{theorem}{Theorem}
\newtheorem{lemma}[theorem]{Lemma}
\newtheorem{corollary}[theorem]{Corollary}

\newtheorem{remark}[theorem]{Remark}

\newcommand{\wi}{\mathrm{Wi\_Index}}
\newcommand{\type}{\mathrm{Type}}
\newcommand{\GGS}{\mathrm{GGS}}
\newcommand{\GTS}{\mathrm{GTS}}
\newcommand{\GMF}{\mathrm{GMF}}

\newcommand{\im}{\mathrm{imm}}
\newcommand{\od}{ \overline{\mathrm{imm}}}

\newcommand{\sP}{  \mathcal{ PO}}
\newcommand{\sO}{  \mathcal{ O}}

\newcommand{\RR}{ \mathbb{R}}
\newcommand{\ZZ}{ \mathbb{Z}}
\newcommand{\QQ}{ \mathbb{Q}}

\newcommand{\CC}{ \mathbb{C}}

\newcommand{\nhalf}{\lfloor n/2 \rfloor}

\newcommand{\comment}[1]{}

\newcommand{\SSS}{\mathfrak{S}}

\newcommand{\id}{\mathsf{id}}

\newcommand{\perm}{  \mathrm{perm}}

\usepackage{mathtools}
\usepackage{footnote}

\begin{document} 
		\title{Laplacian Immanantal Polynomials of  a Bipartite Graph and Graph Shift Operation }
		\author{Mukesh Kumar Nagar}
		\email{mukesh.kr.nagar@gmail.com}
		\address{Department of Mathematics\\
			Jaypee Institute of Information Technology   \\
			Noida, India - 201309.}



\maketitle	

\begin{abstract}
Let $G$ be a bipartite graph on $n$ vertices with the  Laplacian matrix $L_G$. 
When $G$ is a tree, inequalities involving coefficients of immanantal polynomials of $L_G$ are known as we go up $\GTS_n$ 
poset of unlabelled trees with $n$ vertices. We extend $\GTS$ operation on a tree to an arbitrary graph,  
we call it generalized graph shift (hencefourth $\GGS$) operation.
Using $\GGS$ operation, we  generalize these known inequalities associated with trees to bipartite graphs.
Using vertex orientations of $G$, we give a combinatorial interpretation for each coefficient of the Laplacian immanantal polynomial of  $G$ which is used to prove counter parts of Schur theorem and Lieb's conjecture for these coefficients. 
We define $\GGS_n$ poset on  $\Omega_{C_k}^v(n)$, the set of unlabelled unicyclic graphs with $n$ vertices where each vertex of the cycle $C_k$ has degree $2$ except one vertex $v$. 
Using $\GGS_n$ poset on  $\Omega_{C_{2k}}^v(n)$, we solves an extreme value problem of finding the max-min pair in $\Omega_{C_{2k}}^v(n)$ for each coefficient of the generalized Laplacian polynomials. At the end of this paper, we also discuss the monotonicity of the spectral radius and the Wiener index of an unicyclic graph when we go up along $\GGS_n$ poset of $\Omega_{C_k}^v(n)$.
\end{abstract}

	{\bf KEYWORDS : }
	 $\GTS$ poset, bipartite graphs, $\GGS$ operation, vertex orientation,  immanantal polynomial.  \\
{\bf AMS CLASSIFICATION :}	05C05  06A06  15A15  

\section{Introduction}
\label{sec:intro}

For a positive integer $n$, let $\SSS_n$  denote the symmetric group on 
$[n]=\{1, 2, \ldots , n\}$. 
A partition $\lambda$ of $n$ is denoted by $\lambda \vdash n$ and 
it is written using the exponential notation with multiplicities of parts written as exponents. Thus, if $i$ appears $n_i$ times in $\lambda $ then $\lambda=1^{n_1}2^{n_2}3^{n_3}\ldots$ and $n_1+2n_{2}+3n_3+\cdots=n.$ 
For $\lambda \vdash n$, let
$\chi_{\lambda}^{}$ denote  the corresponding irreducible character of  
$\SSS_n$ over $\CC$, the set of complex numbers,  
(see the textbook by Sagan \cite{sagan-book} as a 
reference for the theory of characters of $\SSS_n$).  
Let $A=(a_{i,j}) \in {\mathbb M}_n(\CC)$, 
where  ${\mathbb M}_n(\CC)$ represents the set of all  $n \times n$ matrices with complex entries. 
Then,  the normalized immanant function  of $A$ associated with $\lambda \vdash n$, denoted $\od_{\lambda}(\cdot)$ is defined as 
\begin{equation}
\label{eqn:def_imm}
\od_{\lambda}(A)=\dfrac{1}{\chi_{\lambda}^{}(\id)} \sum_{\psi \in \SSS_n} \chi_{\lambda}^{}(\psi) \prod_{i=1}^{n}a_{i,\psi(i)},
\end{equation}
where $\chi_{\lambda}^{}(\id)$ is the dimension of the
irreducible representation indexed by $\lambda$. The expression   
$\chi_{\lambda}^{}(\id)\od_{\lambda}(A)$
from Equation \eqref{eqn:def_imm}  is called the 
immanant function of $A$ and it is denoted as   $\im_{\lambda}(A)$. 
As $\chi_{1^n}^{}(\id)=1=\chi_{n}^{}(\id)$,   we see that 
$\od_{1^n}(A) = \det(A)$ 
and $\od_{n}(A) = \perm(A)$, where $\det(A)$ and  $\perm(A)$ are the 
determinant and the permanent of $A$, respectively.

Let $ {\mathbb H}_n(\CC)$ be the set of all $n\times n$ positive semidefinite Hermitian matrix over $\CC$. Then,   
Schur in \cite{schur-immanant-ineqs} showed that 
$\det(A)= \min \ \{\od_{\lambda}(A): \lambda\vdash n\}$ 
for all $A \in {\mathbb H}_n(\CC) $. 
Towards getting the maximum element in this set, a popular conjecture known as the ``permanental dominance conjecture'' was given by 
Lieb in~\cite{lieb-permanent-top}. It  states that $\perm (A)= \max \ \{\od_{\lambda}(A): \lambda\vdash n\}$ for all $A \in {\mathbb H}_n(\CC) $. 
This is still open. In this paper, we give a proof of  
this conjecture for the Laplacian matrix of a  
 bipartite graph, see Theorem \ref{thm:Lieb_conj_bip}.

For a simple graph $G$ with vertex set $[n]$, 
its Laplacian matrix $L_G^{}$ is defined by 
$L_G^{}=D-A(G)$, where $A(G)$ is the adjacency matrix of $G$ and 
$D$ is the diagonal matrix with vertex degrees on the main diagonal.  
Indexed by $\lambda \vdash n$, the Laplacian immanantal polynomial of $G$, 
denoted $\phi_{\lambda}^{}(L_G^{},x)$ is defined as $\phi_{\lambda}^{}(L_G^{},x)=\im_{\lambda}(xI-L_G^{})$. 
For $0\leq r \leq n$, let $b_{\lambda,r}(L_{G}^{})$ be the coefficient of 
$(-1)^r x^{n-r}$ in $\phi_{\lambda}(L_{G}^{},x)$, that is  
\begin{equation}
\label{eqn:def_lapl_imm_poly}
 \phi_{\lambda}(L_G^{},x)=\sum_{r=0}^{n} (-1)^r b_{\lambda,r}(L_G^{})x^{n-r}
\end{equation}


Let $T$ be a tree on $n$ vertices with  Laplacian matrix $L_T$. 
Let $S_n$ and $P_n$ be the star tree and the path tree on $n$ vertices respectively. 
When $\lambda=1^n \vdash n$, Gutman and Povlovic in \cite{gutman-pavlovic_laplacian_coeff} conjectured the following inequality 
which was proved by Gutman and Zhou \cite{gutman-zhao-connection-laplacian-spectra} 
and independently by Mohar \cite{mohar-laplacian-coeffs-acyclic-graphs} 
\begin{equation}
\label{eqn:gutman_pav_conj}
b_{1^n,r}^{}(L_{S_n}) \leq b_{1^n,r}^{}(L_{T}) \leq b_{1^n,r}^{}(L_{P_n}) \mbox{ for } r=0,1,2,\ldots,n.
\end{equation} 

The heart of this paper is the $\GGS$ (generalized graph shift) operation 
which is used to study the  monotonicity results of some graph-theoretical parameters. Some of them are discussed here, for instances, the spectral radius, Wiener index and coefficients of the Laplacian immanantal polynomials (see Theorems \ref{thm:main_thm} and \ref{thm:gen_poly_main_thm} and Corollary \ref{cor:min_max_pairs}). 

Kelmans \cite{kelmans} was the first who studied an  operation on graphs called Kelmans transformation, see Definition \ref{def:kelman}.  
This transformation increases the spectral radius 
and decreases the number of spanning trees 
(for more details see  Brown, Colbourn and Devitt \cite{brown-colbourn-devitt} and  Satyanarayana, Schoppman and Suffel \cite{satyanarayana-schoppman-suffel}. 
Similar type of transformation which is called generalized tree shift  
(abbreviated as $\GTS$ henceforth) operation  was defined by Csikv{\'a}ri in \cite{csikvari-poset1} to construct a poset called $\GTS_n$  
on the set of  unlabelled trees with $n$ vertices, see Definition \ref{def:gts}.  
Later, this $\GTS$ operation was studied in \cite{csikvari-poset2,mukesh-eigenvalue,mukesh-siva-immanantal_polynomial,mukesh-siva-GMF} in order to discuss monotonicity properties of some 
graph-theoretical parameters (see Table \ref{tab:max-min}). Csikv{\'a}ri showed  that the star tree  $S_n$ and the path tree 
$P_n$ are the only maximal and minimal elements of $\GTS_n$  respectively.  
Thus for all monotonicity results on $\GTS_n$, 
the max-min pair is either $(P_n,S_n)$ or $(S_n,P_n)$ among all trees with $n$ vertices.
Among other results, he proved  
that going up along $\GTS_n$ decreases each coefficient 
of the characteristic polynomial of $L_T$ in absolute value and 
hence the max-min pair for this algebraic parameter is $(P_n,S_n)$. 
Thus Csikv{\'a}ri's result is more general than the inequality given in \eqref{eqn:gutman_pav_conj}. Using $\GTS_n$ the following stronger 
inequality involving the Laplacian immanantal polynomial of a tree  than the one mentioned above,   
appeared in Nagar and Sivasubramanian \cite[Theorem 1]{mukesh-siva-immanantal_polynomial}. 

\begin{theorem}[Nagar and Sivasubramanian]
	\label{thm:coeff_trees_gts} 
	Let $T$ be a tree on $n$ vertices.
	Then for all $\lambda \vdash n$, going up on $\GTS_n$ decreases each coefficient in absolute value, that is 
	$b_{\lambda,r}(L_T)$ of  the Laplacian immanantal polynomial of $T$ indexed by $\lambda$. 
\end{theorem}

Csikv{\'a}ri's $\GTS_n$ poset and the above results motivated us to extend 
 the notion of generalized tree shift operation  on trees   to  arbitrary graphs. 
 We will call this new operation as the {\it generalized graph shift} 
 (abbreviated as $\GGS$ henceforth), see Definition \ref{def:egts}. 
 The $\GGS$ operation gives us a poset on the set of unlabelled unicylic 
 graphs with certain conditions. 
 The $\GGS$ poset and the 
 following theorem which is our main result, 
 are used to solve an extreme value problem of finding 
 the max-min pair in set of unlabelled unicyclic graphs (see Theorem  \ref{thm:gen_poly_main_thm}  and Corollary  \ref{cor:min_max_pairs}).

\begin{theorem}
	\label{thm:main_thm} 
	Let $G$ be a bipartite graph on $n$ vertices and 
	let $L_G$ be its Laplacian matrix. 
	Then for all $\lambda \vdash n$ and for all $r\geq 0$,  each coefficient 
	$b_{\lambda,r}(L_{G})$ given in \eqref{eqn:def_lapl_imm_poly} 
	is a non-negative integer and the $\GGS$ operation decreases each  $b_{\lambda,r}(L_{G})$.   
\end{theorem} 

Let $G$ be a bipartite graph on $n$ vertices with the Laplacian matrix $L_G$.
Using Theorem \ref{thm:main_thm}, we obtain several corollaries 
involving the constant term of the immanantal polynomials of $L_G$. 
%

The layout of this paper is the following: 
The next section introduces the concepts of Kelmans transformation, 
the generalized tree shift poset and it's generalized version, the $\GGS$ operation on arbitrary graphs. In Section~\ref{sec:imm_poly}, a very basic facts and an elementary 
relationship between  $\im_{\lambda}(L_G)$ and 
the enumerations of vertex orientations in a bipartite graph $G$ are given.  Lieb's conjecture for the Laplacian immanants of a $G$ is also proved. These results are generalized to each coefficient of the Laplacian immanantal polynomial of a bipartite graph in Section \ref{sec:coeff_imm_poly}.
In Section~\ref{sec:proof_main_thm}, we prove Theorem \ref{thm:main_thm} 
which can be thought as a result involving  coefficients of the 
generalized matrix polynomial indexed by the Schur symmetric function, since each irreducible character of $\SSS_n$ is the inverse image of the Frobenius characteristic map of the Schur symmetric function (see Sagan \cite{sagan-book} for more details). 
Further, Theorem \ref{thm:main_thm}  is extended to the generalized matrix polynomial of $L_G$ associated with the elementary, power sum and 
complete homogeneous symmetric functions in Section~\ref{sec:gen_fun_symm_fun}.  
In the last section,  using $\GGS$ operation, 
the monotonicity property of the  spectral radius and 
Wiener index of a graph are discussed.

\section{Graph operations}
\label{sec:def_posets}

 Throughout this paper, our graphs  are simple and connected with vertex set $[n]$. 
The contents of this section, may be conveniently presented into two parts  separately introduce the notion of Kelmans transformation, $\GTS_n$ poset and 
it's generalized version for arbitrary graphs. 
These graph operations are main tools to discuss monotonicity results involving some algebraic and  topological parameters of a graph. 
Using $\GTS$ operation and generalized graph shift,  some of them are determined and are mentioned in Table \ref{tab:max-min}, Theorem \ref{thm:gen_poly_main_thm}  and Corollary \ref{cor:min_max_pairs}.  

\subsection{Kelmans transformation and $\GTS_n$ poset}
\label{subsec:poset_gts_n} 

Towards defining the Kelmans transformation, we need the 
following terminology and notations. 
 For a given vertex $v$ in a graph $G$, define $N[v]$ to be  the set of neighbours of $v$ containing $v$. Define  $N(v):=N[v]\setminus \{v\}$, that is, $N[v]$ is a 
 disjoint union of $v$ and $N(v)$.  We begin with the following definition of Kelmans transformation. 
\begin{definition}
	\label{def:kelman} 
	Let $G_1$ be a graph with $n$ vertices and  
	let $x$ and $y$ be two arbitrary vertices of  $G_1$. 
	We construct a graph $G_2$ by erasing all edges between $x$ and 
	$N(x)\setminus N[y]$ and add edges between $y$ and 
	$N(x) \setminus N[y]$.  This operation is called  Kelmans transformation. 
	We note that the  number of edges in the obtained graph $G_2$ equals the number of edges in $G_1$. 
\end{definition}

Kelmans transformation can be applied to any graph, 
but if we consider it as a transformation on trees to get a connected graph we have to make a restriction on vertices $x$ and $y$ in $G_1$. 
Namely they should have distance at most $2$ in order to obtain a connected graph $G_2$ as a result.  To handle this problem Csikv{\'a}ri \cite{csikvari-poset1} put some restrictions on $x$ and $y$ to define the generalized tree shift operation on trees. 
From \cite{csikvari-poset1}, we recall his definition of $\GTS_n$ poset on the set of unlabelled trees with $n$ vertices.
\begin{definition} 
	\label{def:gts}
	Let $T_1$ be a tree with $n$ vertices. Assume that $1$ and 
	$k$ are two vertices of $T_1$ such that the interior vertices 
	(if they exist) on the unique path $P_{1,k}$ between $1$ and $k$, have degree 2. 
	Let $k-1$ be the neighbour of $k$ on $P_{1,k}$. 
	Construct a new tree $T_2$ by moving all neighbours of $k$ except $k-1$ 
	to the vertex $1$. This operation is called the generalized tree shift. 
	Here, we say  $T_2$ is obtained from $T_1$ using $\GTS$ operation. 
	This is illustrated in Figure \ref{fig:gts_example}. 
	The generalized tree shift operation gives us a partial order  
	denoted as  ``$\leq_{\GTS_n}$" on the set of unlabelled trees on $n$ vertices. 
\end{definition}


 \begin{figure}[ht]
  	\centerline{\includegraphics[scale=0.65]{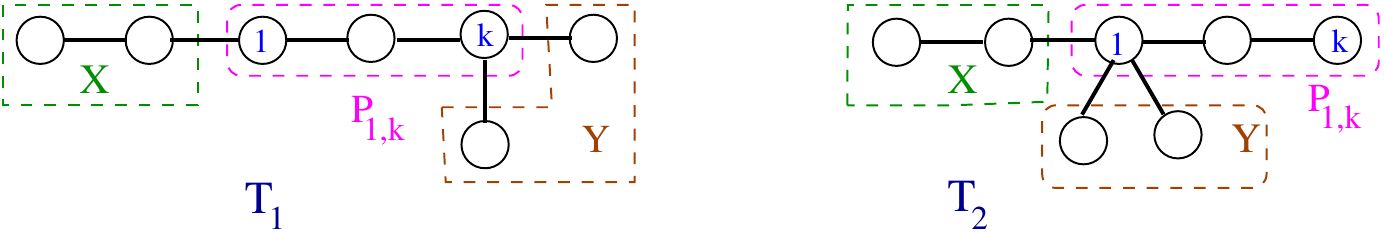}}
  	\caption{Two trees with $T_1 \leq_{\GTS_n} T_2$.}
  	\label{fig:gts_example}
  \end{figure}

If $T_1\leq_{\GTS_n} T_2$, we say tree $T_1$ is below   $T_2$ 
or $T_2$ is above $T_1$. When $n=6$, we refer the reader to Csikv{\'a}ri 
\cite{csikvari-poset1} for the Hasse diagram of $\GTS_6$. 
Among other results he proved the following important result which gives the max-min pair for a monotonicity result on $\GTS_n$ poset.
\begin{lemma}[Csikv{\'a}ri]
	\label{lem:Csikvari_min-max}
Among trees, the star graph $S_n$   and the path graph $P_n$  on $n$ vertices   
are the only maximal and the minimal elements of $\GTS_n$ respectively.
\end{lemma}
  
Thus, for all monotonicity results on $\GTS_n$,  the  max-min pair among all 
unlabelled trees on $n$ vertices is either $(P_n,S_n)$ or $(S_n,P_n)$. 
   For some  well known monotonicity results,  the max-min pair in the set 
   of trees with $n$ vertices,   
are given in Table \ref{tab:max-min}.  When we go up along $\GTS_n$, 
upward ($\uparrow$) and downward ($\downarrow$) arrows  
show graph-theoretical parameters which are increasing and decreasing respectively.  
For instance, in \cite{csikvari-poset2} authors showed that the largest 
eigenvalues of both the matrices, adjacency $A(T)$ and 
the Laplacian $L_T$ increase and hence the max-min pairs for these 
spectral properties are $(S_n, P_n)$. Later in \cite{mukesh-eigenvalue} 
authors generalized these results to the $q$-Laplacian and $q,t$-Laplacian  
matrix for all $q,t\in \RR_{\geq 0}$, the set of non-negative reals. They also proved results involving 
exponential distance matrix and determined the max-min pair for 
the largest and the smallest eigenvalues among the set of unlabelled trees with $n$ vertices.
 
  	\begin{table}[h!]
  		 \begin{center}
  	\begin{tabular}{|c | c | c|}
  		\hline
  	Graph-theoretical parameter & Monotonicity & Max-Min pair  \\ 
  		\hline
  	Number of closed walks of  a fixed length $\ell$ \cite{csikvari-poset1} & $\uparrow$ & $(S_n,P_n)$\\
  	\hline
  		Estrada index \cite{csikvari-poset1}  & $\uparrow$ & $(S_n,P_n)$\\
  	\hline
  		Wiener index  \cite{csikvari-poset1} & $\downarrow$ & $(P_n,S_n)$\\
  	\hline
  		Algebraic connectivity \cite{csikvari-poset2} & $\uparrow$ & $(S_n,P_n)$\\
  	\hline
  		The largest eigenvalue of adjacency and the Laplacian \cite{csikvari-poset2} & $\uparrow$ & $(S_n,P_n)$\\
  	\hline
  		Coefficients of matching polynomial \cite{csikvari-poset2} & $\downarrow$ & $(P_n,S_n)$\\
  	\hline
  		Coefficients of the Laplacian immanantal polynomial \cite{mukesh-siva-immanantal_polynomial} & $\downarrow$ & $(P_n,S_n)$\\
  	\hline
  \end{tabular}
  \end{center}
\caption{Monotonicity of graph theoretical parameters when we go up along  $\GTS_n$.}
	\label{tab:max-min}
  \end{table}

\subsection{Generalized graph shift}
\label{subsec:poset_gts_n_c} 
Inspired by Kelmans transformation and $\GTS_n$ poset, 
an identical definition of the generalized graph shift  operation on a graph $G$ is given in this subsection. 
To define $\GGS$ operation on $G$, some restrictions are used on the chosen vertices $x$ and $y$ of $G$ in 
Kelmans transformation defined above but the $\GTS$ operation on a tree is generalized for an arbitrary graph.  

\begin{definition}
	\label{def:egts}
Let $G_1$ be an arbitrary  graph with $n$ vertices.  
Let  $1$ and $ k $ be two vertices in $G_1$ connected  via a path say $P_{1,k}$ from $1$ to $k$ such that 
\begin{enumerate}
	\item  each 
	interior vertices (if they exist) on $P_{1,k}$  have degree 2 and   
	\item both the vertices $1$ and $k$ are not contained in one cycle of $G_1$.
\end{enumerate}
Thus $k$ has a unique neighbor on $P_{1,k}$, let it be $k-1$. 
Construct a graph $G_2$ by moving all neighbours of $k$ 
except $k-1$ to the vertex $1$. 
 Thus, all neighbors of  $k$ which do not lie on the path 
 $P_{1,k}$ have become neighbors of  $1$ in  $G_2$. 
 This operation 
 is called the  
 generalized graph shift (henceforth $\GGS$), denoted  $G_2=\GGS(G_1)$. 
 Note that the $\GTS$ operation is a special case of $\GGS$ 
 construction when $G_1$ is a tree.
 For the sake of clarity $\GGS$ operation is illustrated in  Figure~\ref{fig:injection_gts_n^c_example}.
\end{definition} 

 \begin{figure}[ht]
	\centerline{\includegraphics[scale=0.65]{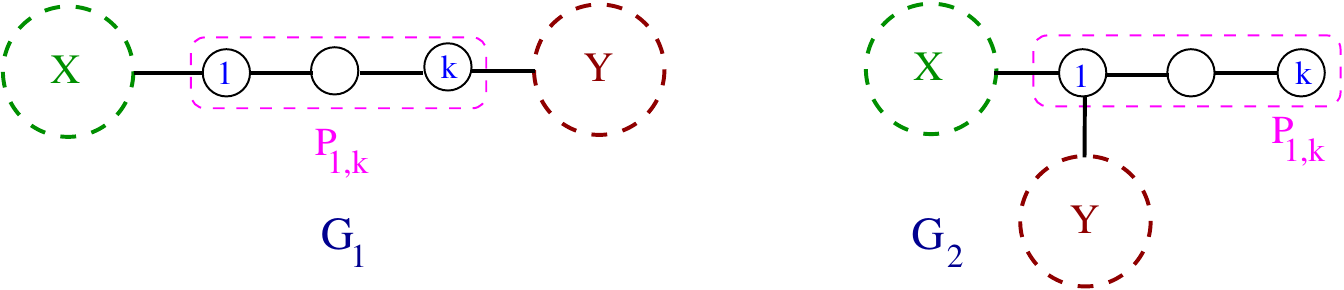}}
	\caption{Two graphs $G_1$ and $G_2$ such that $G_2=\GGS(G_1)$.}
	\label{fig:injection_gts_n^c_example}
\end{figure}

In the above definition, vertices $1$ and $k$ are 
called the recipient and the donor, respectively. 
Here we note that if the role of the recipient and donor  
in $G_1$ are exchanged then the obtained   graph in $\GGS$ operation is  isomorphic to $G_2$. 
It is easy to check that $\GGS$ increases the number of leaf vertices. 
In fact the number of leaves in $G_2$ is one more than the number of leaves in $G_1$ if and only if both $1$ and $k$ are not leaves  in $G_1$.


\subsection{Applications of $\GGS$ operation}
For two positive integers $n$ and $k$ with $n>k\geq 3$, let $\Omega_{C_{k}}^v(n)$ be the set of all unlabelled unicyclic  
graphs with $n$ vertices where the length of the unique  cycle $C_{k}$ is $k$ and  
the degree of each vertex in $C_k$ is $2$ except one  vertex $v\in C_{k}$. We define an order relation, denoted  ``$\leq_{\GGS_n}$'' on $\Omega_{C_{k}}^v(n)$ as follows: If $G_1,G_2\in \Omega_{C_{k}}^v(n)$ and $G_2$ is obtained from $G_1$ using some number of $\GGS$ operations on $G_1$ then  $G_1 \leq_{\GGS_n} G_2$. 
It is easy to check that the relation $\leq_{\GGS_n}$  is  
a partial order  on the set $\Omega_{C_{k}}^v(n)$. 
For the sake of clarity, when $n=8$ and $k=4$  
the Hasse diagram of $\GGS_8$ poset  on $\Omega_{C_{4}}^v(8)$ is given in Figure \ref{fig:hasse_diag_gts_8^4}.  

 \begin{figure}[ht]
	\centerline{\includegraphics[scale=0.9]{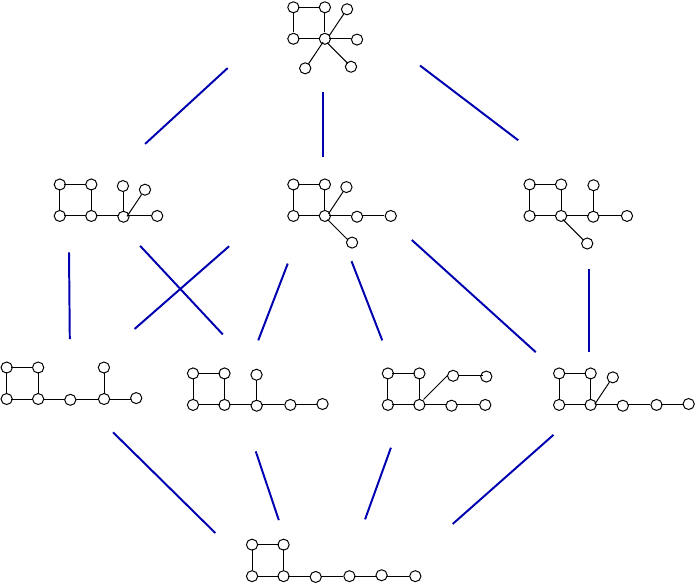}}
	\caption{The Hasse diagram of $\GGS$ poset on the set $\Omega_{C_{4}}^v(8)$.}
	\label{fig:hasse_diag_gts_8^4}
\end{figure}

For a fixed positive integer $k\geq 3$, let $C_k$ be the cycle  on $k$ vertices. 
Let $S_{n-k+1}$ be the star graph on $n-k+1$ vertices such that vertex  
$v_1\in S_{n-k+1}$ has degree $n-k$. Construct a new graph $G_{S_{n-k+1}}$ by moving all neighbours of $v_1$ to a vertex (say $v$) of $ C_k$ and deleting $v_1$. 
This operation of joining $S_{n-k+1}$ to $C_k$ is denoted as $S_{n-k+1}|v_1:C_k|v$. 
Thus $G_{S_{n-k+1}}=S_{n-k+1}|v_1:C_k|v$ is an unicyclic graph with $n$ vertices.
Let  $P_{n-k+1}$ be the path graph on $n-k+1$ vertices such that 
$v_2\in P_{n-k+1}$ is a leaf vertex. Analogously, define $G_{P_{n-k+1}}:=P_{n-k+1}|v_2:C_k|v$ as an unicyclic graph with $n$ 
vertices having cycle $C_k$ of length $k$ and each vertex in $C_k$ 
has degree $2$ except the vertex $v$ which has degree $3$ in $G_{P_{n-k+1}}$. 
It is easy to check that  $G_{S_{n-k+1}},G_{P_{n-k+1}}\in \Omega_{C_{k}}^v(n)$.  
Since the proof of the following lemma is identical to the proof of Lemma 
\ref{lem:Csikvari_min-max},  we omit it and  merely state the result.

\begin{lemma}
\label{lem:gts_n^c_min-max} 
Let $G_{S_{n-k+1}}$ and $G_{P_{n-k+1}}$ be the unicyclic graphs defined in the above paragraph. Then 
 $G_{S_{n-k+1}}$ and $G_{P_{n-k+1}}$ are the only maximal and the minimal elements of $\GGS_n$ poset on $\Omega_{C_{k}}^v(n)$ respectively.

%
\end{lemma}

The above lemma is illustrated in Figure \ref{fig:hasse_diag_gts_8^4}, 
the Hasse diagram of $\GGS_8$ poset on the subset $\Omega_{C_{4}}^v(8)$ of unlabelled unicyclic graphs with $8$ vertices. 
Thus for all monotonicity results on the $\GGS_n$ poset of  $\Omega_{C_{k}}^v(n)$, the max-min pair is either $(G_{S_{n-k+1}},G_{P_{n-k+1}})$ or $(G_{P_{n-k+1}},G_{S_{n-k+1}})$.  
When we restrict $\GGS$ operation to the set of bipartite graphs for  
calculating the Laplacian immanantal polynomial,  Theorem \ref{thm:main_thm} 
determines the max-min pair in the set $\Omega_{C_{2k}}^v(n)$ for the each coefficient $b_{\lambda,r}(L_{G})$ where  $\lambda \vdash n$, $k\geq 2$ and $r=0,1,\ldots,n$.

\begin{remark} 
	\label{rem:arbitrarygraph}
	In the above $\GGS_n$ poset on the set  $\Omega_{C_{k}}^v(n)$, the cycle $C_k$ can 
	be replaced by any connected graph to construct different poset on some subset of connected graphs. 
	For instance, when $C_4$ is replaced by $P_2$ in Figure \ref{fig:hasse_diag_gts_8^4}, 
	we get  the Hasse diagram of $\GTS_6$ given in Csikv{\'a}ri  \cite{csikvari-poset1}. 
\end{remark}



\section{Vertex Orientations and Immanants of Bipartite Graphs}
\label{sec:imm_poly}
It is worth pointing out the  notion of vertex orientation that 
has appeared in several context and having connections 
with the number of matchings, moments of vertices, coefficients of the Laplacian immanantal polynomial  and 
Wiener index of a tree (for more details see  \cite{chan_lam_bipartite,hook_immanant_explained-chan_lam,mukesh-siva-hook,mukesh-siva-immanantal_polynomial}). 
For convenience of the reader, we repeat some relevant material 
on vertex orientation from \cite{chan_lam_bipartite,mukesh-siva-hook,mukesh-siva-immanantal_polynomial} without proof, 
thus making our exposition self contained. 
In this paper, enumeration of vertex orientations will be used to express each  
coefficient of the Laplacian immanantal polynomial of a bipartite graph as a sum of  
non-negative terms. From \cite{chan_lam_bipartite,mukesh-siva-hook} we recall the following definition of vertex orientation of a graph.
\begin{definition}
	\label{def:vertex_orient} 
	Let $G$ be a graph with $n$ vertices. 
	In a vertex orientation $O$ of $G$, we assign an arrow to each vertex 
	$v \in G$ pointing away from $v$ along one of its incidence edges. 
	If $\deg(v)$ denotes the degree of  $v\in G$ then there are $\deg(v)$ choices of assigning an arrow to $v$ in $G$. Such an assignment of arrows to each vertex of 
	$G$ is said to be a vertex orientation.  Throughout this paper, $O(v)=u$ 
	denotes the arrow of $v$ assigned on the edge $(v,u)$, away from $v$ and 
	towards  $u$. It is illustrated in Figure \ref{fig:example of vertex orient}.
\end{definition}

From the above definition of vertex orientation in a graph $G$, it is easy to check that each edge in $G$ can have atmost two arrows on it. 
Let $G$ be a bipartite graph with $n$ vertices and let $O$ be a vertex orientation in $G$. 
Since $G$ is a bipartite graph, no directed cycle with odd length can   
be possible in $O$. 
 As done by Chan and Lam in \cite{chan_lam_bipartite}, if in $O$, there are $n_1$ 
 edges with one arrow on them, $n_2$ bidirected edges (with two arrows on them),  $n_4$ directed $4$-cycles, $\ldots $, $n_{2k}$ directed $2k$-cycles for $k\ge 3$ 
 then it is simple to check that $n_1+2n_2+4n_4+\cdots+2kn_{2k}+\cdots=n$ where  $n_1,n_2,n_4,\ldots,n_{2k},\ldots\in \ZZ_{\geq 0}$, the set of non-negative integers. Each orientation $O$ in a bipartite graph gives us a partition $\mu=(1^{n_1},2^{n_2},4^{n_4},\ldots)$ of $n$. Such an orientation $O$ in $G$ is said to be of type $\mu$ orientation.  Let us define $\sO_{\mu}^{G}$ to be the set of 
 vertex orientations in $G$ of type $\mu$  and let $|\sO_{\mu}^{G}|=a_{G}^{}(\mu)$.   Let $\sP_G(n)$ be the set 
 of all possible partitions of $n$ obtained from the set of orientations in $G$.   
In Figure \ref{fig:example of vertex orient}, three vertex orientations $O$, $P$ and 
$Q$ of a bipartite graph  $G=P_3|1:C_4|1 \in \Omega_{C_4}^1(6)$ are given. 
It is very easy to check that the types of  $O$, $P$ and $Q$ are  $(1^2,2^2)$, $(1^2,4^1)$ and $(2^1,4^1)$, respectively. In $O,P$ and $Q$, edges with two arrows and a 4-cycle are depicted with using brown and blue colours (can seen better on a color monitor), respectively.

\begin{figure}[h]
	\centerline{\includegraphics[scale=0.72]{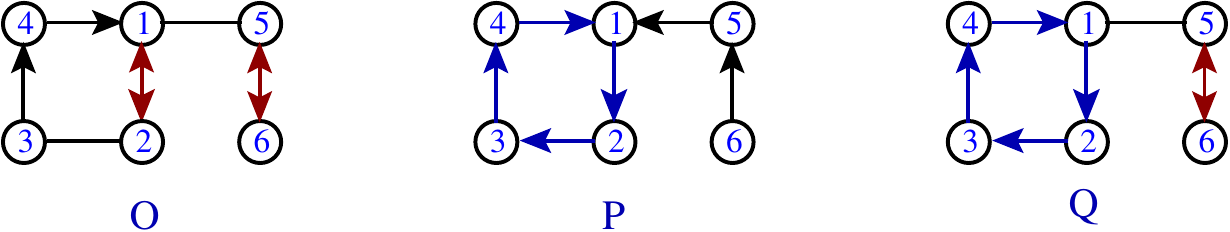}}
	\caption{Three vertex orientations of $C_4|1:P_3|1$.}
	\label{fig:example of vertex orient}
\end{figure}

 Let $P(n)$ be the set of all partitions of $n$.  For $\lambda \vdash n$, let
 $\chi_{\lambda}^{}$ denote  the corresponding irreducible character of  
 $\SSS_n$. In this paper,  $\chi_{\lambda}^{}(\nu)$ will be denoted the character value $\chi_{\lambda}^{}(\psi)$ evaluated at a permutation $\psi\in \SSS_n$ with cycle type $\nu$. 
For $\nu,\mu \in P(n)$ with $\mu=(1^{n_1},2^{n_2},3^{n_3},\ldots,n^{n_n})$ 
and  $\nu=(1^{m_1},2^{m_2},3^{m_3},\ldots,n^{m_n})$, we define 
$\binom{\mu}{\nu}=\binom{n_2}{m_2}\binom{n_3}{m_3}\cdots\binom{n_n}{m_n}$, where $\binom{n_i}{m_i}$ is the binomial coefficient and $\binom{n_i}{m_i}=0$ if $n_i<m_i$. For $\lambda,\mu \vdash n$, 
define $\alpha_{\lambda,\mu}^{}$ as  the following binomial weighted sum involving a character  $\chi_{\lambda}^{}$ of $\SSS_n$ and  binomial coefficients. 
\begin{equation}
\label{eqn:def_alpha}
\alpha_{\lambda,\mu}^{}=\sum_{\nu \in P(n)} \chi_{\lambda}^{}(\nu)\binom{\mu}{\nu}.
\end{equation}

When $\mu=(1^{n_1},2^{n_2},4^{n_4},\ldots)$ 
with $n_i=0$ for all $i>2$,  Chan and Lam   \cite{chan-lam-binom-coeffs-char} showed that   $\alpha_{\lambda,\mu}^{}$ is a  non-negative integral multiple of $2^{n_2}$. 
In \cite[Lemma 3.1]{chan_lam_bipartite}, they also proved the following more general result which will be used to combinatorialise the coefficient $b_{\lambda,r}(L_G)$ of Laplacian immanantal polynomial of the  bipartite graph $G$.

\begin{lemma}[Chan and Lam]
	\label{lem:chan_lam_char_sum_positive}
	For all $\lambda,\mu \in P(n)$, the term $\alpha_{\lambda,\mu}^{}$ defined in \eqref{eqn:def_alpha}, is a  
	non-negative integer. 
\end{lemma}

Using enumeration of vertex orientations and Lemma \ref{lem:chan_lam_char_sum_positive}, Chan and Lam in  \cite[page 5]{chan_lam_bipartite} proved the 
following result in order to determine the Laplacian  immanants of a bipartite graph.
\begin{lemma}[Chan and Lam]
	\label{lem:imm_orient_bip}
	Let $G$ be a bipartite graph on $n$ vertices with $L_G$ as its Laplacian matrix. 
	Then for all  $\lambda \vdash n$,  
	$$\im_{\lambda}(L_{G})=\sum_{\mu \in \sP_G(n)}a_{G}^{}(\mu)\sum_{\nu \in P(n)} \chi_{\lambda}^{}(\nu)\binom{\mu}{\nu} = \sum_{\mu \in P(n)}a_{G}^{}(\mu) \alpha_{\lambda, \mu}^{}
	,$$
	where $a_{G}^{}(\mu)$ is the number of vertex orientations of type $\mu$ in $G$.
\end{lemma} 


From Section \ref{sec:intro} we recall that the following result is the Lieb's conjecture involving  Laplacian immanants 
of a bipartite graph. 

\begin{theorem}[Lieb's Conjecture for bipartite graphs]
	\label{thm:Lieb_conj_bip}
	Let $G$ be a bipartite graph with $n$ vertices and let $L_G$ be its Laplacian matrix. 
	Then for all $\lambda \vdash n$,  $\od_{\lambda}(L_G)\leq \perm(L_G).$
\end{theorem} 

\begin{proof}
	From Lemma \ref{lem:imm_orient_bip}, we have 
	\begin{align*}
	\od_{\lambda}(L_{G})=&\frac{\im_{\lambda}(L_{G})}{\chi_{\lambda}^{}(\id)} =  \sum_{\mu \in \sP_G(n)}a_{G}^{}(\mu)\sum_{\nu \in P(n)} \frac{\chi_{\lambda}^{}(\nu)}{\chi_{\lambda}^{}(\id)}\binom{\mu}{\nu} \\
	& \leq   \sum_{\mu \in P(n)}a_{G}^{}(\mu)\sum_{\nu \in P(n)} \left|\frac{\chi_{\lambda}^{}(\nu)}{\chi_{\lambda}^{}(\id)}\right|\binom{\mu}{\nu}\\
&	 \leq   \sum_{\mu \in P(n)}a_{G}^{}(\mu)\sum_{\nu \in P(n)} \binom{\mu}{\nu}=\perm(L_G).
	\end{align*}
	
	The above identity follows from the fact that for each $\lambda \vdash n$ and 
	for each permutation $\psi\in \SSS_n$, we have  $\left|\frac{\chi_{\lambda}^{}(\psi)}{\chi_{\lambda}^{}(\id)}\right|\leq 1$. 
	Hence the proof is complete.
\end{proof}

\vspace{2mm}
It is well known that for a bipartite graph $G$ with $n$ vertices, $L_G$ is a positive semidefinite matrix. Thus  combining Theorem \ref{thm:Lieb_conj_bip} and Schur's result from Section \ref{sec:intro} gives the  following inequalities.
\begin{equation}
\label{eqn:Lieb_schur_ineq}
\det(L_G^{})\leq \od_{\lambda}(L_{G})\leq \perm(L_G) \ \ \mbox{ for all } \lambda \vdash n.
\end{equation}
It is easy to see that inequalities given in \eqref{eqn:Lieb_schur_ineq} solves an extreme value problem of finding the max-min pair in the set of all partitions of $n$ for the Laplacian immanant of a bipartite graph. Thus the maximum and the minimum value of $\od_{\lambda}(L_{G})$ for a bipartite graph $G$ with $n$ vertices are  attained  at $\lambda=n$ and  $\lambda=1^n$, respectively.

\section{Coefficients of the Laplacian immanantal polynomial}
\label{sec:coeff_imm_poly} 
From Lemma \ref{lem:imm_orient_bip}, enumeration of vertex orientations is used to express $\im_{\lambda}(L_G)$ as a sum of non-negative integers.  
 The important point to note 
here is a  combinatorial interpretation for each coefficient $b_{\lambda,r}(L_{G})$  in the Laplacian immanantal polynomial $\phi_{\lambda}(L_{G},x)$ of  the bipartite graph $G$ 
using vertex orientations. 

Let $G$ be a bipartite graph with vertex set $[n]$ and 
let 
$L_G=(\ell_{i,j})_{1\leq i,j\leq n}$ be the Laplacian matrix of $G$. 
As done in the proof of Lemma \ref{lem:imm_orient_bip} in \cite{chan_lam_bipartite}, $B$-orientations will be used to  express the coefficient  $b_{\lambda,r}(L_G)$ as a sum of non-negative terms. 
Let $B\subseteq [n]$ with $|B|=r$ and let $G_B$ be the subgraph induced by 
$G$ on the set $B$. We orient each vertex $v$ of $B$ only  to one of its neighbor 
(which may or may not belong to $B$) in $G$. Thus  each $v\in B$ has  $\deg(v)$ different choices for assigning an arrow to $v$. 
 Such vertex orientations are known as $B$-orientations in $G$. 

For a fixed $B\subseteq [n]$ with $|B|=r$, let $\sO_{G,B}^{}(\mu)$ be the set of all  $B$-orientations in $G$ which   have type $\mu$.  Define $a_{G,B}(\mu)$ to be the  number of 
elements in $\sO_{G,B}^{}(\mu)$, {\it i.e.}, $a_{G,B}^{}(\mu)=|\sO_{G,B}^{}(\mu)|$. Also define 
$\displaystyle\sO_{G,r}^{}(\mu)=\cup_{B\subseteq [n]  \mbox{ with } |B|=r}^{} \sO_{G,B}^{}(\mu)  \mbox{ and }
a_{G,r}^{}(\mu)=|\sO_{G,r}^{}(\mu)|.$ In other words, one can check that  $a_{G,r}^{}(\mu)=\sum_{B\subseteq [n]  \mbox{ with } |B|=r}^{} |\sO_{G,B}^{}(\mu)|.$ 
Let $\sP_{G,B}(n)$ be the set of all partitions of $n$ obtained from  $B$-orientations in  $G$.  Since $G$ is a bipartite graph, the induced subgraph $G_B$ 
is also a bipartite graph, each $\mu \in \sP_{G,B}(n)$ has  type $\mu=(1^{n_1},2^{n_2},4^{n_4},\ldots)$ for some non-negative integers $n_i$.
 In the proof of the following result,  these terms are used to get  a combinatorial interpretation for the coefficient $b_{\lambda,r}(L_G)$ as a sum of non-negative integers.  

\begin{lemma}
	\label{lem:coeff_non-negative}
	Let $G$ be a bipartite graph on $[n]$ with the Laplacian matrix $L_G$. Then, for all $\lambda \vdash n$ and 
	for $0\leq r \leq n$,  
	$$b_{\lambda,r}(L_G)=\sum_{\mu \in P(n)} a_{G,r}^{}(\mu)\alpha_{\lambda,\mu}^{}.$$ 
	In particular, $b_{\lambda,r}(L_G)$ is a non-negative integer for all $\lambda \vdash n$ and for $r=0,1,\ldots,n.$
\end{lemma}
\begin{proof}
	For a subset $B\subseteq [n]$  with  $|B|=r$, define 
	$b_{\lambda,B}(L_G)=\im_{\lambda}
	\left[
	\begin{matrix}
	L_{G}[B|B] & 0 \\
	0 & I
	\end{matrix}
	\right],
	$ 
	where $L_{G}[B|B]$ is a submatrix of $L_{G}$ induced by the rows and columns 
	with indices in $B$ and $I$ is the $(n-r)\times (n-r)$ identity matrix. 
	It is easy to check  that $b_{\lambda,r}(L_G)=\sum\limits_{B\subseteq [n]  \mbox{ with } |B|=r} b_{\lambda,B}(L_G).$

	We recall that $G$ is a bipartite graph and the induced subgraph $G_B$ 
	is also a bipartite graph. 
  Thus it is natural to see that a permutation $\psi \in \SSS_n$ with cycle type $\mu$ contributes in $b_{\lambda,B}(L_G)$ if and only if there is a $B$-orientation $O\in \sO_{G,B}(\mu)$ and $\psi $ fixes $[n]-B$. 
	For a given $\nu=(1^{m_1},2^{m_2},4^{m_4},\ldots) \vdash n$,   
	define the set 
	$S_{\nu}^B:=\{(\psi,O): \prod_{i\in B} \ell_{i,\psi(i)}\neq 0,\type(\psi)=\nu, \psi $ fixes $[n]-B$ and $O\in \sO_{G,B}(\mu)$ with $\mu=(1^{n_1},2^{n_2},4^{n_4},\ldots) \in \sP_{G,B}(n)$ 
	such that 
	$n_i\geq m_i$ for all $i\geq 2\}$. For a given $\psi\in \SSS_n$, cycles in $\psi$ with size $2$ represent bidirected edges in $O$ and each cycle of $\psi$ 
	with size strictly greater than $2$ gives a directed cycle in $O$. 
	Each even cycle and each fixed point $v$  of $\psi$ contribute $1$ and $\deg(v)$ in  $\prod_{i\in B} \ell_{i,\psi(i)}$, respectively. This product of degrees gives the number of   
	$B$-orientation $O$ of type $\mu\in \sP_{G,B}(n)$ such that $n_i\geq m_i$ for 
	all $i\geq 2$.  Therefore,   the number of elements in $S_{\nu}^B$ equals  
	to the sum of $\prod_{i\in B} \ell_{i,\psi(i)}$ over all permutations which fix $[n]-B$ and 
	whose cycle type is $\nu$. 
	 
	We can also enumerate the set $S_{\nu}^B$ in another way.  
	Take a $B$-orientation $O$ of type 
	$\mu=(1^{n_1},2^{n_2},4^{n_4},\ldots) \in\sP_{G,B}(n)$ with $n_{i}\geq m_i$ for all 
	$i\geq 2$.  Therefore we have  $\binom{n_2}{m_2}$ ways to pick  $m_2$ bidirected edges and $\binom{n_4}{m_4}$ ways to pick $m_4$   directed $4$-cycles and so on. 
	So there are $\binom{n_2}{m_2}\binom{n_4}{m_4}\cdots=\binom{\mu}{\nu}$ 
	ways to construct the ordered pair $(\psi,O)$. Thus we get

	$$|S_{\nu}^B|=\sum_{\type(\psi)=\nu, \psi \mbox{ fixes } [n]-B}\prod_{i\in B}\ell_{i,\psi(i)} = 
	 \sum_{\mu\in \sP_{G,B}(n)}\binom{\mu}{\nu}a_{G,B}^{}(\mu).$$

	By multiplying $\chi_{\lambda}^{}(\nu)$ and by summing over all $B$ with $|B|=r$, the above expression gives us the coefficient  $b_{\lambda,r}(L_G)$ of $(-1)^rx^{n-r}$ in $\phi(x,L_{G})$ involving the  term  $a_{G,r}^{}(\mu)$ as follows 
	\begin{align*}
		b_{\lambda,r}(L_G) & =  \sum_{B\subseteq [n]  \mbox{ with } |B|=r} \sum_{\nu \in P(n)} \chi_{\lambda}^{}(\nu) \sum_{\type(\psi)=\nu, \psi \mbox{ fixes } [n]-B} \prod_{i\in B} \ell_{i,\psi(i)}. \\
		& =  \sum_{B\subseteq [n]  \mbox{ with } |B|=r} \sum_{\nu \in P(n)} \chi_{\lambda}^{}(\nu) \sum_{\mu\in \sP_{G,B}(n)}\binom{\mu}{\nu}a_{G,B}^{}(\mu) \\
		& =  \sum_{\mu\in P(n)}a_{G,r}^{}(\mu) \sum_{\nu \in P(n)} \chi_{\lambda}^{}(\nu) \binom{\mu}{\nu},
	\end{align*}
where the last equality follows from exchanging the order of summations. 
Now,  using \eqref{eqn:def_alpha} and Lemma \ref{lem:chan_lam_char_sum_positive} the proof is completed. 
\end{proof}

\vspace{2mm}
From the above proof, we recall that $$b_{\lambda,r}(L_G)=\sum\limits_{B\subseteq [n]  \mbox{ with } |B|=r} b_{\lambda,B}(L_G), \mbox{ where }
b_{\lambda,B}(L_G)=\im_{\lambda}
\left[
\begin{matrix}
	L_{G}[B|B] & 0 \\
	0 & I
\end{matrix}
\right].$$
  Thus using  Lemmas \ref{lem:chan_lam_char_sum_positive} and  
\ref{lem:coeff_non-negative}, the proof of the following theorem is identical to the proof of  \eqref{eqn:Lieb_schur_ineq} and hence we omit it and merely state the result. 

\begin{theorem}
	\label{thm:gen_lieb_conj}
	 For all $\lambda \vdash n$ and for all bipartite graph $G$ on $n$ vertices, we have 
	$$
	b_{1^n,r}(L_G)\leq 
	\frac{b_{\lambda,r}(L_G)}{\chi_{\lambda}^{}(\id)}\leq b_{n,r}(L_G) \mbox{ for } r=0,1,2,\ldots,n.$$
\end{theorem} 

When $r=n$, Theorem \ref{thm:Lieb_conj_bip} is a special case of Theorem \ref{thm:gen_lieb_conj}. 
Thus using Theorem \ref{thm:gen_lieb_conj} for a bipartite graph $G$ and for $0\leq r\leq n$  the maximum and the minimum value of $\frac{b_{\lambda,r}(L_G)}{\chi_{\lambda}^{}(\id)}$  are attained at $\lambda=n$ and $\lambda=1^n$, respectively.

\section{Proof of Theorem \ref{thm:main_thm}}
\label{sec:proof_main_thm}

Let $T_1$ and $T_2$ be two trees with $n$ vertices such that $T_2$ 
covers $T_1$ in $\GTS_n$ poset. 
For a given $B\subseteq [n]$ and for $i=1,2$, let $\sO_{T_i, B}(\mu)$ be the set of all $B$-orientations with type $\mu$ in $T_i$. 
For $0\leq r \leq n$, let $\sO_{T_i, r}(\mu)=\bigcup_{B\subseteq [n] \mbox{ with } |B|=r} \sO_{T_i, B}(\mu)$.
Nagar and Sivasubramanian  \cite{mukesh-siva-immanantal_polynomial} 
used the following result involving $\sO_{T_i, r}(\mu)$ in   proving Theorem \ref{thm:coeff_trees_gts}. 
In this paper, we shall use an identical result to prove our main result. 
\begin{theorem}[Nagar and Sivasubramanian]
	\label{thm:mukesh_injection}
Let $T_1$ and $T_2$ be two trees on $n$ vertices such that $T_2$ covers 
$T_1$ in $\GTS_n$. 
Then for each type $\mu=2^j,1^{n-2j}$ with $0\leq j \leq \nhalf$ and for $0\leq r \leq n$, there is an injective map 
$\theta:\sO_{T_2, r}(\mu)\rightarrow \sO_{T_1, r}(\mu)$.
\end{theorem}


Let $G_1$ and $G_2$ be two bipartite graphs with $n$ vertices such that 
$G_2=\GGS(G_1)$, where $1$ and $k$ are  the recipient and the donor in $G_1$, respectively. 
Let $P_{1,k}$ be the unique path between the vertices $1$ and $k$ which witness this $\GGS$ operation. 
Then $G_1$ and $G_2$ can be illustrated in Figure \ref{fig:injection_gts_n^c_example}, 
for some subgraphs $X$ and $Y$ of both $G_1$ and $G_2$. 
For the convenience we will assume $X$ and $Y$ as vertex subsets  of $[n]$ and 
the following union and intersection are on these subsets.
For a subset $B\subseteq [n]$, define $B'=(B\cap X)\cup(B\cap Y)\cup B_P^{t}$, 
where $B_P^t$ contains vertex $k+1-i$ if $B\cap P_{1,k}$ contains the vertex $i$ of $P_{1,k}$. 
Since the proof of the following lemma is similar to the proof of 
Theorem  \ref{thm:mukesh_injection}, we only sketch our proof. 
\begin{lemma} 
	\label{lem:lem_injection1}
	Let $G_1$ and $G_2$ be two bipartite graphs with $n$ vertices such that $G_2=\GGS(G_1)$. 	Then for all $r \geq 0$ and for all $\mu \in P(n)$  
 there is an injective map $\delta:\sO_{G_2, r}(\mu)\rightarrow \sO_{G_1, r}(\mu)$.	
\end{lemma}

\begin{proof} 
	Let $G_1$ and $G_2$ be two graphs as given in Figure \ref{fig:injection_gts_n^c_example}. 
	For a given $B\subseteq [n]$, let $O \in \sO_{G_2, B}(\mu)$.  
	We first consider the case when either  $1\notin B$ or $O(1)\notin Y$. 
	We define $O'\in \sO_{G_1, B}(\mu)$ as follows. In $O'$, for each vertex 
	$v\in B$, assign the same orientation as it is given in $O$. 
	Therefore, the type $\mu$ for both orientations $O$ and $O'$ is identical. 
	Thus, we get $O'=\delta(O)\in \sO_{G_1, B}(\mu)$ if   $O\in \sO_{G_2, B}(\mu)$ and 
	$O'=\delta(O)\in \sO_{G_1, B'}(\mu)$ if   $O\in \sO_{G_2, B'}(\mu)$. 
	
	We next consider the remaining case, \it{i.e.}, when $1\in B$ and 
	$O(1)=y$, for some $y \in Y$. 
	For each $O \in \sO_{G_2, B}(\mu)$, we define $O'\in \sO_{G_1, B'}(\mu)$ as follows. 
	We orient $y'\in Y$ such that $O'(y')=k$, whenever $y'\in B$ and $O(y')=1$. 
	For each $v\in B-(P_{1,k}\cup\{y\})$, assign the same orientation in 
	$O'$ to $v$ as it is given in $O$. We also assign orientation to the vertex $k$ in $G_1$ as 
	$O'(k)=y$. 
	If  vertex $i \in P_{1,k}$ with $i\neq 1$ is oriented ``towards vertex $1$'' in $O$, then orient vertex $k+1-i$ 
	``towards $k$'' in $O'$ and 
	likewise if $i$ is oriented ``away from 1'' then orient $k+1-i$ ``away from 
	$k$'' in $O'$.   
	It is illustrated in Figure \ref{fig:inj_orient_example} where $P_{1,k}=P_{1,6}$ 
	and $B\cap P_{1,k}=\{1,3,4,5,6\}$. 
	It is very easy to check that the type $\mu$ of $O$ is unchanged in $O'$. 
	One can check that $O'=\delta(O)\in \sO_{G_1, B'}(\mu)$ if   $O\in \sO_{G_2, B}(\mu)$ 
	and 
	$O'\in \sO_{G_1, B}(\mu)$ if   $O\in \sO_{G_2, B'}(\mu)$.  
	Therefore we get an injective map from $O\in \sO_{G_2, B}(\mu)\cup O\in \sO_{G_2, B'}(\mu) $ 
	to 
	$\sO_{G_1, B}(\mu)\cup \sO_{G_1,B'}(\mu)$. 
	Thus   an injective map $\delta:\sO_{G_2, r}(\mu)\rightarrow \sO_{G_1, r}(\mu)$ is constructed by extending the above map to the set $\sO_{G_2, r}(\mu)$ and hence the proof is completed.
\end{proof}

 \begin{figure}[h]
	\centerline{\includegraphics[scale=0.64]{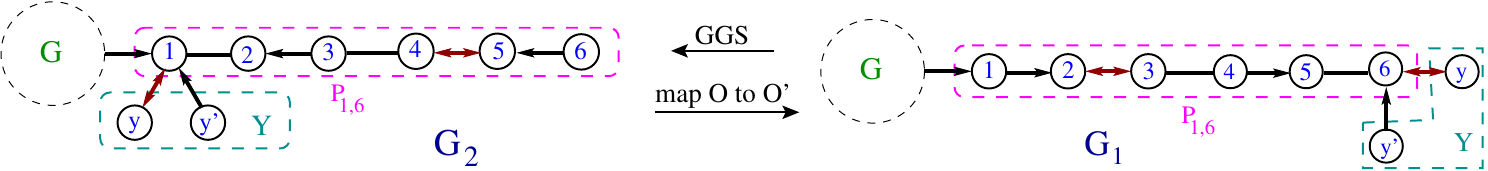}}
	\caption{An example of oriented vertices mapped from  $B\cap P_{1,k}$ to $B'\cap P_{1,k}$.}
	\label{fig:inj_orient_example}
\end{figure}

\vspace{2mm}
Enumerating the sets $\sO_{G_2, r}(\mu)$ and $\sO_{G_1, r}(\mu)$ using Lemma \ref{lem:lem_injection1}, we get $|\sO_{G_2, r}(\mu)| \leq  |\sO_{G_1, r}(\mu)|$ for all $r\geq 0$ and for all $\mu \in P(n)$. Thus combining Lemmas \ref{lem:chan_lam_char_sum_positive} and \ref{lem:coeff_non-negative} with Lemma \ref{lem:lem_injection1} 
gives  a proof of the following result which gives Theorem \ref{thm:main_thm} as an easy consequence.
\begin{theorem} 
\label{thm:main_thm1}
Let $G_1$ and $G_2$ be two bipartite graphs with $n$ vertices such that $G_2=\GGS(G_1)$.	
Then for all $r \geq 0$ and for all $\mu \in P(n)$, we assert that 
$a_{G_2,r}^{}(\mu)\leq a_{G_1,r}^{}(\mu).$ Consequently, we have 
$b_{\lambda,r}(L_{G_2})\leq b_{\lambda,r}(L_{G_1})$  for all $\lambda \vdash n$.
\end{theorem}

Thus, $\GGS$ operation decreases each coefficient of the Laplacian 
immanantal polynomial of a bipartite graph. 
Also taking $r=n$ in Theorem \ref{thm:main_thm1} concludes that the $\GGS$ operation on a bipartite graph decreases each  Laplacian immanant $\im_{\lambda}(L_G)$ indexed by $\lambda$. Further using Lemma \ref{lem:gts_n^c_min-max} and Theorem  \ref{thm:main_thm1}, 
we solve the following extreme value problem of finding the max-min pair in $\Omega_{C_{2k}}^v(n)$ for each  $b_{\lambda,r}(L_{G})$
\begin{corollary}
	Going up along $\GGS$ poset of $\Omega_{C_{2k}}^v(n)$ decreases 
	each coefficient of the Laplacian immanantal polynomial of a unicyclic graph in absolute value. 
	Consequently, the max-min pair for these coefficients in absolute value is   $(G_{P_{n-2k+1}},G_{S_{n-2k+1}})$  in the set  $\Omega_{C_{2k}}^v(n)$.
\end{corollary}

\section{Generalized matrix function indexed by symmetric functions}
\label{sec:gen_fun_symm_fun}

For a positive integer $n$, let  $\SF_{n,\QQ}$ denote the vector space of degree 
$n$ symmetric functions over $\QQ$, the set of rational numbers. Thus, 
each symmetric function  $\tau $ of  $\SF_{n,\QQ}$ can be  written as a linear 
combination of basis vectors of $\SF_{n,\QQ}$.
In the theory of algebraic combinatorics which  involves symmetric functions, 
the vector space $\SF_{n,\QQ}$ has six standard bases.
In this paper, standard terminology is used to 
denote each of the usual bases of $\SF_{n,\QQ}$.  Thus 
 for a given $\lambda \vdash n$, the
Schur,  power sum, elementary, homogeneous, monomial 
and forgotten symmetric functions are denoted by  $s_{\lambda}$, $p_{\lambda}$, $e_{\lambda}$, $h_{\lambda}$, 
$m_{\lambda}$ and $f_{\lambda}$, respectively. 
We refer the reader to the textbooks by
Stanley \cite{EC2} and  by Mendes and Remmel \cite{mendes-remmel-book} 
for background on 
symmetric functions.  It is well known that $\SF_{n,\QQ}$ has an inner product structure as well. Another inner product space often studied is $\CF_n$,
the space of class functions from $\SSS_n \mapsto \CC$.  Further, there is a well
known isometry between these two spaces called the Frobenius characteristic, denoted  
$\ch: \CF_n \rightarrow \SF_{n,\QQ}$
(see \cite{EC2} for more details).

Let $T$ be a tree on $n$ vertices with Laplacian matrix $L_T$.  
From Sections \ref{sec:intro}, \ref{sec:imm_poly} and \ref{sec:proof_main_thm}, we recall that inequalities  involving coefficients of the Laplacian 
immanantal polynomial of  $T$ are known as we go up  along $\GTS_n$ poset. 
For each $\lambda\vdash n$, it is well known that the inverse Frobenius
image $\ch^{-1}_{s_{\lambda}}$ of the Schur symmetric function $s_{\lambda}^{}$ 
is $\chi_{\lambda}^{}$,
the irreducible character of $\SSS_n$ over $\CC$ indexed by $\lambda$. 
Thus using the Frobenius characteristic map $\ch(\cdot)$, this can be thought as a result associated with the 
Schur symmetric function $s_{\lambda}^{}$ (see Nagar and Sivasubramanian   \cite{mukesh-siva-GMF} for more details). In \cite{mukesh-siva-GMF} authors 
introduced a generalized matrix function on the set of square matrices associated with 
an arbitrary symmetric function. Further  they proved that going up $\GTS_n$ 
decreases the absolute value of each coefficient of generalized 
matrix polynomial associated with 
$s_{\lambda}$, $p_{\lambda}$, $e_{\lambda}$, $h_{\lambda}$, 
$m_{\lambda}$ and $f_{\lambda}$ for each partition $\lambda$ of $n$. 

For $\tau \in \SF_{n,\QQ}$, let $\ch^{-1}_{\tau} = \ch^{-1}(\tau)$ be 
the inverse Frobenius image of $\tau$.  Clearly, $\ch^{-1}_{\tau} \in \CF_n$ is a class function on  $\SSS_n$ over $\CC$
indexed by $\tau$. From \cite{mukesh-siva-GMF} we recall the following definition of 
the generalized matrix function ($\GMF$ henceforth) of  a matrix  
$A=(a_{i,j}) \in {\mathbb M}_n(\CC)$ associated with $\tau$. It is denotes as $\GMF_{\tau}(A)$ and it is  defined by  
\begin{equation}
	\label{eqn:gen_matrix_fn}
	\GMF_{\tau}(A) = \sum_{\psi \in \SSS_n} \ch^{-1}_{\tau}(\psi) \prod_{i=1}^n a_{i,\psi(i)}.
\end{equation}

Using \eqref{eqn:gen_matrix_fn}, when $\tau=s_{\lambda}$ it is simple to see that   $\displaystyle \GMF_{\tau}(A)=\im_{\lambda}(A)$.  
Inspired by immanantal polynomial, the generalized matrix polynomial of a matrix $A$   associated with $\tau$ denoted as  $\phi_{\tau}^{}(A,x)$ is defined by $\phi_{\tau}^{}(A,x)=\GMF_{\tau}(xI-A)$. Thus when $\tau=s_{\lambda}$, we have $\GMF_{\tau}(xI-A)=\im_{\lambda}(xI-A)$,  an immanantal polynomial of $A$ indexed by $\lambda$. 
In particular, when $\lambda=1^n$ and  $\tau=s_{\lambda}$, $\GMF_{\tau}(xI-A)$ is the characteristic polynomial of $A$.

Let $G$ be a graph with $n$ vertices and let   $L_G=(\ell_{i,j})_{1\leq i,j\leq n}$ be its Laplacian matrix. 
Thus from \eqref{eqn:gen_matrix_fn}, the generalized matrix function  
$\GMF_{\tau}(L_G)$ of $L_G$ associated with $\tau \in \SF_{n,\QQ}$ is given by  

\begin{equation*}
	\label{eqn:gen_matrix_fn_sL}
	\GMF_{\tau}(L_G) = \sum_{\psi \in \SSS_n} \ch^{-1}_{\tau}(\psi) \prod_{i=1}^n \ell_{i,\psi(i)}.
\end{equation*}
and the generalized matrix polynomial of  $L_G$ indexed by $\tau $ is given  by $\phi_{\tau}^{}(L_G,x)=\GMF_{\tau}(xI-L_G)$. It is very easy to show the following
counterpart of Lemma  \ref{lem:imm_orient_bip}.  Since
the proof is a verbatim copy, we omit it.

\begin{theorem}
	\label{thm:mukesh-siva-gen_matrix_fn}
	Let $L_G$ be the Laplacian matrix of a bipartite graph $G$ with $n$ vertices. 
	Then for all $\tau \in \SF_{n,\QQ}$,  we assert that  
	$\displaystyle \GMF_{\tau}(L_G) = \sum\limits_{\mu\in P(n)} \alpha_{\mu} ^{}(\tau) a_{G}^{}(\mu)$, where 
	$\displaystyle\alpha_{\mu}^{}(\tau)=\sum\limits_{\nu \in P(n)} \ch^{-1}_{\tau}(\nu)\binom{\mu}{\nu}.
	$ 
\end{theorem}

Since $\chi_{\lambda}^{}=\ch^{-1}_{s_{\lambda}}$,
 we have  $\displaystyle \alpha_{\mu}^{}(\tau)=\alpha_{\lambda,\mu}^{}$ and hence $\displaystyle \GMF_{\tau}(L_G)=\im_{\lambda}(L_G)$. 
It is very simple to show the following
counterpart of Lemma \ref{lem:coeff_non-negative}.  Since
the proof is identical, we omit it and merely state the result.

\begin{lemma}
	\label{lem:coeff_gmf_poly}
	Let $G$ be a bipartite graph on $n$ vertices with Laplacian matrix $L_G$. 
	Then, for all $\tau \in \SF_{n,\QQ}$, we have 
	$\displaystyle \GMF_{\tau}(xI-L_G)=\sum\limits_{r=0}^n (-1)^r 
	b_{\tau,r}^{}(L_{G}) x^{n-r}$, where 
	$\displaystyle b_{\tau,r}^{}(L_{G})
	=\sum\limits_{\mu\in P(n)}\alpha_{\mu}^{}(\tau)a_{G,r}^{}(\mu)$ and $r\in\{0,1,\ldots,n\}$.
\end{lemma}

As mentioned earlier, the inverse Frobenius
image $\ch^{-1}_{s_{\lambda}}$ of the Schur symmetric function $s_{\lambda}$ 
is $\chi_{\lambda}^{}$,
the irreducible character of $\SSS_n$ over $\CC$ indexed by $\lambda$.
Thus, if any symmetric function $\tau \in \SF_{n,\QQ}$ is Schur-positive
(that is, $\tau = \sum_{\lambda \vdash n} c_{\lambda} s_{\lambda}$ where
$c_{\lambda} \in \RR^{+}$ for all $\lambda \vdash n$), 
then, by linearity, Lemma \ref{lem:chan_lam_char_sum_positive}  
will be true with $\chi_{\lambda}^{}$ replaced by $\ch^{-1}(\tau)$ in the definition of $\alpha_{\lambda,\mu}^{}$ in \eqref{eqn:def_alpha}.   
Since $h_{\lambda}$ and $e_{\lambda}$
are Schur-positive (see \cite[Corollary 7.12.4]{EC2}), it follows from
Theorem \ref{thm:mukesh-siva-gen_matrix_fn} that 
$\alpha_{\mu}^{}(h_{\lambda}) \geq 0$ and $\alpha_{\mu}^{}(e_{\lambda}) \geq 0$ for all $\lambda, \mu \vdash n$.  
%
Since the inverse Frobenius image $\ch^{-1}_{p_{\lambda}}$ of
$p_{\lambda}$ is a scalar multiple of the indicator function of 
the conjugacy class $C_{\lambda}$ indexed by $\lambda$ (see \cite{EC2}), it follows
that $\alpha_{\mu}^{}(p_{\lambda}) \geq 0$ for all $\lambda,\mu \vdash n$.
We record these well known facts below for future use. 

\begin{lemma}
	\label{lem:non_neg_alpha_tau}
For $\lambda, \mu \in P(n)$ and $\tau \in \{e_{\lambda},p_{\lambda},h_{\lambda} \}$, 
let	$\alpha_{\mu}^{}(\tau)$ be as used  in Theorem \ref{thm:mukesh-siva-gen_matrix_fn}.  
 Then $\alpha_{\mu}^{}(\tau)$ is a non-negative integer.
\end{lemma}

For all $\lambda,\mu \vdash n,$ it is very simple to check that the above lemma is not true when $\tau=m_{\lambda}^{}$, the monomial symmetric function associated with $\lambda$. When $n=4$ the values of the quantity $\alpha_{\mu}^{}(m_{\lambda})$ are tabulated in Table \ref{tab:monomial}. When  $\mu=4$ from Table \ref{tab:monomial}, one can see that   $\alpha_{\mu}^{}(m_{1^4}^{})=-1$, $\alpha_{\mu}^{}(m_{2^2}^{})=-2$  and $\alpha_{\mu}^{}(m_{3,1}^{})=-4$. 
Although when $\mu=2^j,1^{n-2j}$, Nagar and Sivasubramanian 
\cite{mukesh-siva-GMF} proved that $\alpha_{\mu}^{}(m_{\lambda})$ is a 
non-negative integral multiple of $2^j$ for each $\lambda$ and for all $j\geq 0$.

	
	\begin{table}[ht]
	\begin{center}
			$$\begin{array}{|l | c | c|c |c|c|}
				\hline
					&   \mu=1^4 & \mu=2,1^2 & \mu=2^2 &\mu=3,1& \mu=4\\ 
					\hline  
			\lambda=1^4 &		1 &  0  & 0  & 0 & -1 \\ 
				\hline
			\lambda=2,1^2&		 0  & 0  & 0 &  0  & \ \ \ 4 \\
				 \hline
			\lambda=2^2 &	 0 &  0  & 0  &  0 &  -2 \\
				 \hline
			\lambda=3,1&	 0  &  0 &   0  &  0  & -4  \\
				 \hline
		\lambda=4 &	0  &  0 &  0  &  0  &  \ \ \ 4 \\
			\hline
				\end{array} $$
			\end{center}
			\caption{The values of $\alpha_{\mu}^{}(m_{\lambda})$, where $\lambda,\mu \vdash 4.$}
				\label{tab:monomial}
	\end{table}


Here, we extend Theorem \ref{thm:main_thm} to the 
generalized matrix polynomials of the Laplacian matrix $L_G$ of a bipartite graph $G$ associated with three standard 
bases involving $e_{\lambda}$, $h_{\lambda}$ and $p_{\lambda}$ for all $\lambda \vdash n$. In other words, 
we consider the cases when we replace the Laplacian 
immanantal polynomial $\im_{\lambda}(xI - L_{G})$ of a bipartite graph $G$ by 
$\GMF_{\tau}(xI - L_{G})$ in Theorem \ref{thm:main_thm}, where  
$\tau \in \{e_{\lambda}, h_{\lambda},p_{\lambda} \}$.  
Thus using the $\GGS$ operation on a bipartite graph, we have the following consequence of  Lemmas \ref{lem:lem_injection1}, \ref{lem:coeff_gmf_poly} and  \ref{lem:non_neg_alpha_tau} and Theorem \ref{thm:main_thm}.

\begin{theorem}
	\label{thm:gen_poly_main_thm}
	For $0 \leq r \leq n$ and for all $\lambda \vdash n$,  
	the $\GGS$ operation on a bipartite graph $G$ decreases the coefficient  $b_{\tau,r}^{}(L_{G})$  of 
	$(-1)^r x^{n-r}$ in $\GMF_{\tau}(xI-L_G)$, where $\tau \in \{s_{\lambda},e_{\lambda},p_{\lambda},h_{\lambda} \}$.  
	Moreover for this monotonicity result on $\GGS$ poset of $\Omega_{C_{2k}}^v(n)$, 
	the max-min pair of unicyclic graphs in  $\Omega_{C_{2k}}^v(n)$ is $(G_{P_{n-2k+1}},G_{S_{n-2k+1}})$.
\end{theorem}

Thus the above result solves an extreme value problem involving coefficients of the generalized matrix polynomial of the Laplacian of a bipartite graph associated with the Schur, elementary, power sum and homogeneous  symmetric functions.

\section{Corollaries}
\label{sec:additional_results}
For the convenience of the reader, from \cite{csikvari-poset1,csikvari-poset2}, 
we repeat some proofs which go through for an arbitrary graph under $\GGS$ operation 
thus making our exposition self contained.  The important point to note here is a natural extension of  monotonicity results on $\GTS_n$  to all connected graphs under $\GGS$ operation and it is used to solve extreme value problems of finding the max-min pair in the set $\Omega_{C_{k}}^v(n)$ for some graph-theoretical parameters which are monotonic under this transformation. 
Let $G$ be a graph with adjacency matrix $A(G)$ and spectral radius $\sigma(G)$ (the largest eigenvalue of $A(G)$).  
 Since the proof of the following theorem is identical to the  result given in 
  \cite[Theorem 3.1]{csikvari-poset2}, 
  we give an outline of the proof. 
\begin{theorem}
	\label{thm:spectral_radius}
	The $\GGS$ operation on an arbitrary graph increases the spectral radius.
\end{theorem}

\begin{proof}
	Let $G_1$ and $G_2$ be two arbitrary graphs with $n$ vertices such that $G_2=\GGS(G_1)$. 
	Let $1$ and $k$  be the recipient and the donor in $G_1$, respectively. 
	Let $\sigma(G_1)$ and $\sigma(G_2)$ be spectral radii of $G_1$ and $G_2$, respectively. Let $w=(w_1,w_2,\ldots,w_n)^t$ be the Perron eigenvector of $A(G_1)$ associated to $\sigma(G_1)$ with $||w||=1$. 
	Recall that if  the role of $1$ and $k$ in $G_1$ are exchanged then the resulting graph after using $\GGS$ transformation is isomorphic to $G_2$. Thus without loss of generality, we assume that  
	$w_1\geq w_k.$ Let $Z_{G_1}(k)=\{v\in G_1:v\in N(k) \mbox{ with } v\neq k-1\}$, 
	where $N(k)$ is the set of neighbours of $k$ which doesn't contain $k$.  Then it is easy to check that 
	\begin{align*}
	\sigma(G_1) & =w^tA(G_1)w= w^tA(G_2)w-2(w_1-w_k)\sum_{j\in Z_{G_1}(k)} w_j\\
	& \leq w^tA(G_2)w \leq \max_{||z||=1} z^tA(G_2)z = \sigma(G_2).
	\end{align*}
	Hence $\sigma(G_1)\leq \sigma(G_2)$ completes the proof.
\end{proof}

\vspace{2mm}
The main point to note about the above theorem is that it is also true when  $G_2=\GGS(G_1)$ and all edges of $G_1$ as well as of $G_2$ are  replaced by blocks (complete graphs). 
%

Let $G$ be a connected graph with $n$ vertices. The Wiener Index of $G$, denoted $\wi(G)$  
is defined by $\wi(G)=\sum_{i,j\in G}d_G(i,j)$, where $d_G(i,j)$ is the shortest 
distance between the vertices $i$ and $j$ in $G$. Since the proof of the following theorem is similar to the proof given in \cite[Theorem 3.1]{csikvari-poset1} we sketch our proof for the sake of completeness.
\begin{theorem}
	\label{thm:wiener_index}
Let $G_1$ and $G_2$ be two connected graphs with $n$ vertices such that 
$G_2=\GGS(G_1)$. Then $\wi(G_1)\geq \wi(G_2)$. 
\end{theorem}
\begin{proof}
Let $d_{G_1}(i,j)$ be the shortest distance between $i$ and $j$ in $G_1$ and similarly $d_{G_2}(u,v)$ is defined for $u,v\in G_2$.  Let $X$ and $Y$ be two subgraphs of $G_1$ as given in Figure  \ref{fig:injection_gts_n^c_example}.
Thus using Figure  \ref{fig:injection_gts_n^c_example}, it is easy to see that 
$$ d_{G_1}(i,x)+d_{G_1}(k+1-i,x)= d_{G_2}(i,x)+d_{G_2}(k+1-i,x) \mbox{ for all } x \in X$$

and 
$$ d_{G_1}(i,y)+d_{G_1}(k+1-i,y)= d_{G_2}(i,y)+d_{G_2}(k+1-i,y) \mbox{ for all } y \in Y. $$

Obviously $ d_{G_1}(x_1,x_2)= d_{G_2}(x_1,x_2)$ for all  $x_1,x_2 \in X$ and  
$ d_{G_1}(y_1,y_2)= d_{G_2}(y_1,y_2)$ for all $ y_1,y_2 \in Y$. It is also easy to check that 
$d_{G_1}(x,y)= d_{G_2}(x,y)+(k-1) \mbox{ for all } x\in X \mbox{ and } y \in Y$ and hence $$\sum_{x\in X, y\in Y}d_{G_1}(x,y)=\sum_{x\in X , y\in Y} d_{G_2}(x,y)+(k-1)|X||Y|.$$

After combining all the above identities, we assert that  
  $ \sum_{i,j\in G_1}d_{G_1}(i,j) \geq \sum_{u,v\in G_2}d_{G_2}(u,v)$. 
  Thus $\wi(G_1)\geq \wi(G_2)$ completes the proof. 
	\end{proof}

\vspace{2mm}
Thus each monotonic result under  $\GGS$ operation is a natural extension of $\GTS$ transformation. 
 Using the definition of $\GGS$ poset on the set  $\Omega_{C_k}^v(n)$, the following result is an easy consequence of  Lemma  \ref{lem:gts_n^c_min-max} and Theorems \ref{thm:spectral_radius} and 
 \ref{thm:wiener_index}.
 
\begin{corollary}
	\label{cor:min_max_pairs}
Let $G_{S}= G_{S_{n-k+1}}$ and   $G_{P}= G_{P_{n-k+1}}$ be the maximal and the minimal elements in  $\Omega_{C_k}^v(n)$, respectively. Then 
from Theorems \ref{thm:spectral_radius} and \ref{thm:wiener_index} with Lemma \ref{lem:gts_n^c_min-max}, $(G_S,G_P)\in \Omega_{C_k}^v(n)$ is the max-min pair for the spectral radius while $(G_P,G_S)\in \Omega_{C_k}^v(n)$ is the max-min pair for the Wiener index.
\end{corollary}

In the above corollary if $C_k$ is replaced by an arbitrary connected graph $H$ with $k$ vertices, then  using Remark \ref{rem:arbitrarygraph} we have a $\GGS_n$ poset on the set $\Omega_{H}^v(n)$ of unlabeled connected graphs with $n$ vertices. This poset also has monotonicity properties for the graph-theoretic parameters  given in Corollary \ref{cor:min_max_pairs}. 
Hence for these parameters, the max-min pairs can be determined in the set $\Omega_{H}^v(n)$.


Associated to a tree are different notion of Randic index, Geometric index and Atom Bond Connectivity (ABC) index. It would be nice to see the behavior of these graph-theoretic indices under the $\GGS$ operation. 


\section*{Acknowledgement}
Theorem \ref{thm:main_thm} in this work was in its conjecture form, 
tested using an open-source computer package ``SageMath''. 
We thank the authors for generously releasing the computer package SageMath as an open-source package.

\bibliographystyle{acm}
\bibliography{main}
\end{document}